\documentclass[12pt, a4paper, english]{article}

\usepackage{babel}
\usepackage[latin1]{inputenc}
\usepackage{amsmath}
\usepackage{amssymb}
\usepackage{amsthm}

\theoremstyle{plain}
\newtheorem{Thm}{Theorem}[section]
\newtheorem{Prop}[Thm]{Proposition}

\newtheorem{Lemma}[Thm]{Lemma}

\theoremstyle{definition}

\newtheorem{Def*}{Definition}

\newcommand \dd {\,{\rm d}}

\begin{document}

\title{Replica Symmetry and Combinatorial Optimization}

\author{Johan W\"astlund \thanks{The hospitality of the Mittag-Leffler institute, Djursholm, Sweden, and the Theory Group at Microsoft Research, Redmond, USA is gratefully acknowledged.}\\
\small Department of Mathematics\\[-0.8ex]
\small Chalmers University of Technology,\\[-0.8ex] \small S-412 96 Gothenburg, Sweden\\[-0.8ex]
\small \texttt{wastlund@chalmers.se}
}
\date{\small \today
}

\maketitle

\begin{abstract}

We establish the soundness of the replica symmetric ansatz introduced by M. M\'ezard and G. Parisi for minimum matching and the traveling salesman problem in the pseudo-dimension $d$ mean field model for $d\geq 1$. The case $d=1$ of minimum matching corresponds to the $\pi^2/6$-limit for the assignment problem established by D.~Aldous in 2001, and the analogous limit for the $d=1$ case of TSP was recently obtained by the author with a different method. 

We introduce a game-theoretical framework by which we prove the correctness of the replica-cavity prediction of the corresponding limits also for $d>1$.
\end{abstract}

\section{Introduction and background} \label{S:intro}
\subsection{An example} Suppose that the edges of the complete graph $K_n$ for even $n$ are assigned independent costs from uniform distribution on $[0,1]$, and that we study the minimum total cost of a perfect matching, in other words a set of $n/2$ edges having no vertex in common. Based on the non-rigorous replica method of statistical physics, M.~M\'ezard and G.~Parisi conjectured in \cite{MP85} that the minimum cost converges in probability to $\pi^2/12$. This was proved by D.~Aldous \cite{A01}, and it is known that the limit depends only on the density of the cost-distribution at zero. Hence apart from a scaling factor, the $\pi^2/12$-limit is stable under replacing uniform $[0,1]$ by various other distributions such as exponential, minimum of two independent $[0,1]$-variables etc.

On the other hand there are distributions that are not of this form, but where the density at $x>0$ scales like a power of $x$ as $x\to 0$. Suppose for instance that the edge costs are sums of two independent uniform $[0,1]$-variables. In that case \cite{MP85} predicted that the minimum cost of a perfect matching is approximately $0.8086\cdot \sqrt{n}$. 

In this paper we obtain a rigorous confirmation of this prediction. We establish the correctness of the replica method for this and some related problems, notably the traveling salesman. We prove that the characterization in \cite{MP85} of the limits (under rescaling by the appropriate power of $n$) is essentially correct.

\subsection{Background}
It has been known for some decades that methods of the statistical mechanics of disordered systems apply to certain problems of combinatorial optimization. Much of the work in this direction stems directly or indirectly from G.~Parisi's solution \cite{P80} of the Sherrington-Kirkpatrick model \cite{SK75} of spin glasses, established rigorously by M.~Talagrand \cite{T06}. In \cite{KT85}, S.~Kirkpatrick and G.~Toulouse suggested the mean field traveling salesman problem (TSP) as an archetypal optimization problem sharing important features with spin glasses. M.~M\' ezard and G.~Parisi \cite{MP85, MP86a, MP86b, MP87, P86} and M\'ezard and W.~Krauth \cite{KM89} obtained several remarkably detailed predictions about minimum matching and the TSP with the replica and cavity methods. These predictions were based on the assumption of \emph{replica symmetry} which is known to fail at low temperatures for models of spin glasses. It became clear that minimum matching and the TSP are different in this respect from models like the Sherrington-Kirkpatrick model and random $k$-SAT. Several authors have verified the consistency of the replica symmetric ansatz by testing its various predictions numerically and theoretically \cite{BKMP91, BM89, BP99,CBBMP97, HMM98, MPV87, PPR03, PR01, P97, PM98, S86}.

Replica symmetry is interesting from an algorithmic point of view since it is linked to the efficiency of Belief Propagation heuristics, see for instance \cite{MPZ02}. Recently J.~Salez and D.~Shah \cite{SS09} have obtained rigorous results in this direction for the assignment problem, and in view of our results, their conclusions should be valid also for the TSP. 

In this paper we introduce a two-person game which is played on a graph with lengths associated to the edges. We show that on certain infinite graphs with random edge-lengths, this game has an almost surely well-defined game-theoretical value, and we argue that this property is essentially equivalent to replica symmetry for the minimum matching problem. 

The paper is organized as follows: In Section~\ref{S:model} we give further background and state our main result on minimum matching, Theorem~\ref{T:main}. The rest of Section~\ref{S:matching} is devoted to the proof of this theorem. The most important new results are the introduction of the \emph{Graph Exploration} game in Section~\ref{S:GE} and the analysis of this game on the $d$-PWIT in Sections~\ref{S:dPWIT}--\ref{S:branching}. In Section~\ref{S:TSP} we state and prove the analogous theorem for the TSP, Theorem~\ref{T:mainTSP}. The TSP corresponds to a \emph{comply-constrain} version of Graph Exploration. We discuss some important differences in the analysis of this game compared to its normal counterpart, but avoid repeating arguments that are similar to those of Section~\ref{S:matching}. In Section~\ref{S:edgeCover} we study the minimum \emph{edge cover}. This problem has not been considered in the physics literature, but has a rich structure and can be analyzed by the same methods as matching and the TSP. Section~\ref{S:variations} contains some concluding remarks.

\section{Minimum matching}\label{S:matching}
\subsection{The mean field model and the minimum matching problem} \label{S:model}
The \emph{mean field model of distance} is a complete graph on $n$ vertices whose edges are assigned i.i.d lengths (or costs) $l_{i,j}$ from a distribution on the positive real numbers. If the edge-lengths are intended to model distances between random points in $d$-dimensional space, then we expect $P(l<r)\sim c\cdot r^d$ for small $r$, since the probability of a point being within distance $r$ of another should be proportional to the volume of a ball of radius $r$. Certain asymptotical properties of optimization problems are known to depend (apart from trivial scaling factors) only on the parameter $d$, so that for instance uniform distribution on $[0,1]$ is equivalent to exponential distribution of mean 1, both belonging to the case $d=1$. Similarly the sum and the maximum of two independent uniform $[0,1]$ variables both represent $d=2$, and are equivalent apart from a scaling factor.

Such results can often be established with standard techniques, but for convenience we make a specific choice of distribution for $0<d<\infty$ by taking $l$ to be the $d$-th root of an exponential variable. To simplify the scaling in terms of $n$, we let $l = (nX)^{1/d}$, where $X$ is exponential of mean 1. In other words $l$ is the $d$-th root of an exponential variable of mean $n$. This gives $$P(l<r) = P((nX)^{1/d}<r) = P(X<r^d/n) = 1 - \exp(-r^d/n) \sim r^d/n.$$ We can thereby regard the lengths $l_{i,j}$ as generated from an underlying set of independent mean 1 exponential variables $X_{i,j}$.
  
A favorite problem is minimum matching, which seems to be the simplest problem that allows the ideas of \cite{MP85} to be displayed in a nontrivial way. We ask for a set of edges of minimum total length such that each vertex is incident to exactly one. This obviously requires $n$ to be even unless we allow one vertex to be left out of the pairing, but this is a minor issue since we are mainly interested in the large $n$ asymptotics. 

The quantity of main interest is the total length $M_n$ of the minimum matching. It is not difficult to guess roughly how $M_n$ scales with $n$. From an arbitrary vertex, the order of the distance to the nearest neighbors is obtained by setting $P(l<r) \approx 1/n$, which leads to $r\approx 1$. If we believe that edge-lengths of order 1 will dominate the solution, then since a solution contains $n/2$ edges, we expect $M_n$ to scale like $n$.

It is natural to conjecture that $M_n/n$, which can be interpreted as the average cost per vertex of obtaining a matching, converges in probability to a constant depending on $d$. Our main result is that this is true for $d\geq 1$.

\begin{Thm} \label{T:main} For every $d\geq 1$ there is a number $\beta_M(d)$ such that 
\begin{equation} \label{mainConjecture} \frac{M_n}{n} \overset{\rm p}\to \beta_M(d).\end{equation}
\end{Thm}

We believe that in principle the method applies also when $0<d<1$, but we have run into some difficulties that have so far prevented us from establishing \eqref{mainConjecture} in that case.

Let us immediately state the easiest available bounds on $\beta_M(d)$. For a lower bound we observe that the expected length of an edge in a matching must be at least the expectation of the length $l_{\min}$ of the shortest edge from a given vertex. A factor $1/2$ comes from averaging the total length of the $n/2$ edges over $n$ vertices: \begin{multline} \label{lowerBound} \beta_M(d) \geq \frac12E(l_{\min}) = \frac12\int_0^{\infty}P(l_{\min}>r)\dd r = \frac12\int_0^{\infty} \exp(-r^d/n)^{n-1}\dd r \\ \sim \frac12\int_0^{\infty} \exp(-r^d)\dd r = \frac12\Gamma(1+1/d).\end{multline} 
Getting an upper bound is not trivial, but it is known \cite{A01} that $\beta_M(1)=\pi^2/12$. For $d\geq 1$ the concavity of the mapping $X\mapsto (nX)^{1/d}$ gives the bound $$\beta_M(d) \leq \frac12 \left(\frac{\pi^2}6\right)^{1/d}.$$ A calculation in \cite{A01} backed up by results in \cite{NPS05} gives the sharper bound \begin{equation} \label{upperBound} \beta_M(d) \leq \frac12\int_0^\infty x^{1/d}\cdot\frac{e^{-x}(e^{-x}-1+x)}{(1-e^{-x})^2}\dd x = \frac1{2d}\cdot\Gamma(1+1/d)\cdot\zeta(1+1/d),\end{equation} valid for $d>0$ (see also \cite{VM09} and \cite{W06path}). A discussion of how to prove \eqref{upperBound} would be out of place here, but the idea is to use the matching that minimizes the sum of the underlying exponential variables $X_{i,j}$. 

It was observed in \cite{VM84} that a greedy algorithm gives a matching of the right order of magnitude if $d>1$ (actually \cite{VM84} considered the TSP, but leaving out every other edge of a tour obviously gives a matching). The greedy matching (also known as the Gale-Shapley matching) gives the weaker bound $$\beta_M(d) \leq \frac12\cdot\frac{\pi/d}{\sin(\pi/d)}.$$ but is interesting in itself.

Within the framework of the replica method, M\'ezard and Parisi obtained an analytical characterization of $\beta_M(d)$ which is conjectured to be correct for all $d>0$. They arrived at an integral equation which is equivalent to \begin{equation} \label{integralEquation} F(x) = \exp\left(-d\int_0^\infty l^{d-1}F(l-x)\dd l\right),\end{equation} from which $\beta_M(d)$ is obtained as $$\beta_M(d) = \frac{d^2}2 \underset{\substack{ -\infty<x,y<\infty \\ x+y\geq 0\\}}{\int\int} (x+y)^{d-1}F(x)F(y)\dd x\dd y.$$
The method is inherently non-rigorous, and it has not been established that \eqref{integralEquation} has a unique solution except in the case $d=1$, where the solution $$F(x) = \frac1{1+e^x}$$ leads to $\beta_M(1)=\pi^2/12$.

M\'ezard and Parisi also calculated $\beta_M(2)$ numerically. In \cite{MP85}, the edge-lengths have density $l^{d-1}$ rather than $dl^{d-1}/n$, which means that for $d=2$, our edges are longer than those of \cite{MP85} by a factor $\sqrt{n/2}$. On the other hand M\'ezard and Parisi considered a graph on $2N$ vertices but rescaled the total length of the matching by a power of $N$, in this case $\sqrt{N}$ rather than $\sqrt{2N}$. Therefore the ground state energy $\hat{E} \approx 1.144$ given in equation (24) of \cite{MP85} is $2\beta_M(2)$. It is interesting to compare their numerical value to the bounds \eqref{lowerBound} and \eqref{upperBound}. For $d=2$ we get $$\frac14\sqrt{\pi} = 0.443\dots\leq \beta_M(2) \leq \frac18\sqrt{\pi}\cdot \zeta(3/2) = 0.578\dots,$$ while \cite{MP85} gives $\beta_M(2) \approx 0.572$.
The fact that the estimated true value is quite close to the upper bound indicates that the matching that minimizes the sum of the underlying exponential variables is not too far from the actual optimum.

It is worth pointing out how the scaling works in a couple of simple examples. Suppose we take the distribution of edge lengths as $l = \max(U, V)$, where $U$ and $V$ are independent and uniform in $[0,1]$. Then $$P(l<r) = P(U<r)^2 = r^2$$ if $r\leq 1$, and therefore the distribution belongs to the case $d=2$. We determine the value of $r_0$ for which the expected number of points within distance $r_0$ from a given point is equal to 1. Approximately this happens when $P(l<r_0)\sim 1/n$, which gives $$r_0\sim \frac1{\sqrt{n}}.$$ We can think of $\beta_M(2)$ as the average cost per vertex of the minimum matching, measured with $r_0$ as the unit of length. Hence the total length of the minimum matching is approximately $$\beta_M(2)\sqrt{n}.$$

If on the other hand we take the edge lengths to be distributed like $U+V$ as in the introduction, then $P(l<r)\sim r^2/2$, and the unit of length is given by $r_0^2/2 = 1/n$ or equivalently $$r_0 = \frac{\sqrt{2}}{\sqrt{n}}.$$ In that case the asymptotical total length of the minimum matching is $$\beta_M(2)\sqrt{2n}.$$

Although it does not follow from Theorem~\ref{T:main}, our results apply also to the \emph{assignment problem}, in other words minimum matching on the complete bipartite graph $K_{n,n}$. In the bipartite model the nearest neighbor distances are the same and the only difference is that a matching contains $n$ edges instead of $n/2$. With the two distributions $\max(U,V)$ and $U+V$, the minimum assignments will have lengths approximately $2\beta_M(2)\sqrt{n}$ and $\beta_M(2)\sqrt{8n}$ respectively.

A more precise value of $\beta_M(2)$ was obtained by J.~Houdayer, J.~H.~Boutet de Monvel and O.~C.~Martin \cite{HMM98}. Using length-distributions normalized by the volume of the $d$-dimensional ball, they obtained the value $0.322580$ for the limit. After sorting out the normalization, one finds that our $\beta_M(2)$ is $\sqrt{\pi}$ times their value, which gives $\beta_M(2) \approx 0.571758$. 
Through an approximate solution of \eqref{integralEquation} we have obtained the value $$\beta_M(2) \approx  0.57175904959888.$$ We have no theoretical estimate of the error, but these decimals seem to be stable. The values equivalent to $\beta_M(d)$ for integers $d\leq 10$ are given in Table~2 of \cite{HMM98}. Although we still do not know whether \eqref{integralEquation} has a solution, the numerical result can be regarded as an approximation of the fixed point of $V_\theta$ (see Section~\ref{S:integralEquation}) for an appropriately chosen $\theta$, and therefore apart from the numerical error the result is backed up rigorously. 
 
On the mathematical side there has been considerable progress on the case $d=1$. In particular the $\pi^2/6$-limit in the assignment problem has received several different proofs \cite{A01, LW04, NPS05, Weasy}. From our point of view the result corresponds to the statement that for $d=1$, the limit in \eqref{mainConjecture} exists and $\beta_M(1)=\pi^2/12$, but the asymptotic equivalence between assignment and matching on the complete graph is by no means trivial, and does not follow from \cite{LW04, NPS05, Weasy}. The proofs together provide a quite detailed picture of the distribution of the total length as well as the local statistics of the optimum solution, and the analogous result for the TSP was established in \cite{W09}. However, the proofs in \cite{LW04, NPS05, Weasy, W09} are very different from the approach in the physics literature, and do not seem to generalize to $d\neq 1$. The original proof by David Aldous \cite{A01} is the one that comes closest to justifying the replica symmetric ansatz (particularly in view of additional results in \cite{B06, SS09}), but it seems to rely on finding a solution to \eqref{integralEquation}. 

In the present paper, our aim is to show that the calculations in \cite{MP85} are sound for quite general reasons. We prove that for $d\geq 1$, \eqref{mainConjecture} holds, and we characterize $\beta_M(d)$ analytically in terms of certain integral equations similar to \eqref{integralEquation}. Although we cannot find explicit solutions to these equations when $d\neq 1$, our results show that the numerical computation of $\beta_M(2)$ in \cite{MP85} is correct in principle. 

For $d=1$, much more detailed results can be obtained. A more precise analysis of the $d=1$ case with the present method, and a clarification of its relation to the results of \cite{W09}, will be given in joint work with G.~Parisi (manuscript in preparation). 

Our approach is ``zero temperature'', but similar to the replica-cavity method in that we reach the optimum solution through a limiting process. We introduce a parameter $\theta$ and study ``diluted'' problems where partial matchings are allowed but penalized by $\theta/2$ for each unmatched vertex. The original problem is recovered in the limit $\theta\to\infty$. The parameter $\theta$ plays a role similar to the inverse temperature in statistical physics. Finite $\theta$ allows for a certain local freedom that destroys all long-range interactions. In particular, adding or deleting a vertex has only a local effect on the optimum solution. In \cite{MP85} a similar assumption seems to be crucial for the renormalization that leads to \eqref{integralEquation}.  

\subsection{Graph Exploration} \label{S:GE}
The following two-person zero-sum game was invented in an attempt to find a mathematically sound interpretation of \eqref{integralEquation}. We call it \emph{Graph Exploration} since it somehow centers around the question whether it is worth the price to be the first to explore a new part of the graph. We are given a graph with nonnegative edge lengths, a starting point $v$, and a nonnegative parameter $\theta$. Alice and Bob take turns choosing the next edge of a self-avoiding walk, with Alice starting the game from $v$. The player who makes a move pays the length of the edge to the opponent. At each turn, the moving player also has the option to, instead of moving, terminate the game by paying $\theta/2$ to the opponent. Each player tries to maximize their total payoff. 

Notice that there is no randomness in the game. The players are assumed to have perfect information about the graph including the edge-lengths. We can immediately make some observations:
\begin{itemize}
\item If the graph is finite, then there is a well-defined game-theoretical value.

\item If the graph is infinite, there may or may not be such a value. For instance, if all edges have the same length $l<\theta$, then no player will ever want to terminate the game.

\item Edges of length more than $\theta$ are irrelevant to the game. If Alice moves along such an edge, then Bob can terminate the game, and even though this may not be Bob's best option, it would still have been better for Alice to terminate in the first place.
\end{itemize} 

\subsection{The diluted matching problem}
There is a relaxation of the minimum matching problem that we refer to as the \emph{diluted} matching problem. Instead of requiring each vertex to be covered by the matching, we allow for any partial matching, with a penalty of $\theta/2$ for each vertex that is not matched. 

For the moment we regard the parameter $\theta$ as fixed. If $G$ is a finite graph with given edge lengths, we let $M(G)$ be the cost of the diluted matching problem. More precisely, $M(G)$ is the minimum, taken over all partial matchings, of the sum of the edge lengths in the matching plus $\theta/2$ times the number of unmatched vertices.

\begin{Prop} \label{P:payoff}
Let $G$ be a finite graph with given edge lengths, and let $v$ be a vertex of $G$ chosen as the starting point for Graph Exploration. Then Bob's payoff under optimal play is $$M(G) - M(G-v).$$
\end{Prop}

\begin{proof}
Suppose that the neighbors of $v$ are $v_1,\dots. v_k$, and that the edges from $v$ to these neighbors have lengths $l_1,\dots, l_k$. Let $f(G, v)$ be Bob's payoff when the game starts at the vertex $v$. By minimizing over Alice's move options, we recursively characterize $f$ by $$f(G, v) = \min(\theta/2,\, l_i - f(G-v, v_i)).$$
On the other hand, the cost of the diluted matching problem satisfies $$M(G) = \min(\theta/2+M(G-v), l_i + M(G - v - v_i)).$$ Subtracting $M(G-v)$ from both sides, we see that $$M(G) - M(G-v) = \min(\theta/2, \,l_i - (M(G-v) - M(G - v - v_i))).$$ This shows that $f(G, v)$ and $M(G) - M(G-v)$ satisfy the same recursion, and it follows by induction that they are equal.
\end{proof}

It is clear from Proposition \ref{P:payoff} and its proof that Alice achieves optimal payoff by starting along the edge of the optimal diluted matching, if there is such an edge from $v$, and otherwise by terminating immediately. By induction it follows that consistently playing along edges of the optimum diluted matching, and terminating when no such edge is available, is minimax optimal. Therefore under mutual optimal play, the path described by the game is the symmetric difference of the optimal diluted matchings on $G$ and $G-v$. Actually the argument provides a simple proof of the fact that this symmetric difference is a path.

Since the diluted matching problem can be solved efficiently by standard matching algorithms, it follows that Graph Exploration can be played optimally with a polynomial time algorithm, but from our perspective this is beside the point. The advantage of introducing the game is that if the graph is infinite there may still be a well-defined game-theoretical value. This value then replaces $M(G) - M(G-v)$ and allows for the equivalent of the renormalization argument of \cite{MP85} in a mathematically consistent way.   

\subsection{Approximation by the PWIT} \label{S:PWITlike}
The Poisson Weighted Infinite Tree (PWIT) was introduced by Aldous \cite{A92, A01}. The PWIT is a rooted tree where each vertex has a countably infinite set of children, and the edges to these children are assigned lengths given by a rate 1 Poisson point process on the positive real numbers (independent processes for all vertices). The PWIT is a \emph{local weak limit} of the mean field model, a statement which has been made precise in slightly different ways in the literature. We establish a simple version which is convenient for our purpose. We first treat the case $d=1$, and later establish an easy refinement valid for general $d$. Recall that for $d=1$, the edges in the mean field model are exponential of mean $n$.

For given $\theta$ and a positive integer $k$, the \emph{$(k, \theta)$-neighborhood} of a vertex $v$ in a graph is the subgraph defined as the union of all paths from $v$ of at most $k$ edges, each of length at most $\theta$.

\begin{Lemma} \label{L:coupling} Suppose $k$ and $\theta$ are given. Consider the graph $K_n$ with a specified root $v$, and a random process consisting in assigning independent exponential lengths of mean $n$ to the edges. There is a coupling of this process to the PWIT such that with probability at least $$1 - \frac{(\theta+2)^k}{n^{1/3}},$$ the $(k, \theta)$-neighborhoods of $v$ and of the root of the PWIT are isomorphic, with corresponding edges having equal length. \end{Lemma}

\begin{proof}
We start from a PWIT rooted in a vertex $v'$, and assign lengths to the edges of $K_n$ through a random mapping of the $(k, \theta)$-neighborhood of $v'$ to $K_n$. We start by mapping $v'$ to $v$. Then we sequentially map the vertices of the $(k, \theta)$-neighborhood of $v'$ to independent uniformly chosen vertices of $K_n$ through a tree search (say depth-first). If we ever choose the root $v$ or a vertex that has already been chosen, then we let the procedure fail.

To see that this is compatible with the probability measure on the edge lengths of $K_n$, define an \emph{extended model} in the following way: For each pair of vertices in $K_n$ there is an infinite sequence of edges whose lengths are given by a  rate $1/n$ Poisson process on the positive reals. Moreover, for each vertex there is a sequence of loops whose lengths are also given by a rate $1/n$ process (hence this model differs slightly from the \emph{friendly model} of \cite{W09}). The original model is then recovered by discarding all loops and all edges except the shortest one between each pair of vertices. Now we explore the $(k, \theta)$-neighborhood of $v$ in the extended model through a depth-first search from $v$, and ``fail'' if that neighborhood turns out not to be a tree.

We want to estimate the probability of failure. Let $N$ be the number of vertices in the $(k, \theta)$-neighborhood of $v'$ (including $v'$). Then  $$E(N) = 1 + \theta + \theta^2 + \dots + \theta^k.$$
Conditioning on $N$, the expected number of collisions is $$\frac{\binom{N}{2}}{n} \leq \frac{N^2}{n}.$$ Therefore the probability of at least one collision is at most $N^2/n$.
Now \begin{multline}P(\text{failure}) \leq P(\text{failure}\,|\,N\leq n^{1/3}) + P(N>n^{1/3}) \\ \leq \frac{n^{2/3}}{n} + \frac{1+\theta+\theta^2+\dots+\theta^k}{n^{1/3}}  = \frac{2+\theta+\theta^2+\dots+\theta^k}{n^{1/3}} \leq \frac{(\theta+2)^k}{n^{1/3}}.\end{multline}
\end{proof}

Lemma~\ref{L:coupling} can easily be generalized to the neighborhoods of several vertices, with the same method of proof:

\begin{Lemma}\label{L:couplingGeneral}
Suppose $m$ vertices in $K_n$ are chosen independently of the edge lengths. Then with probability at least $$1 - \frac{(m\theta+m+1)^k}{n^{1/3}},$$ the union of their $(k,\theta)$-neighborhoods is isomorphic to a disjoint union of the $(k,\theta)$-neighborhoods of the roots of $m$ independent PWITs.
\end{Lemma}

For general $d$, we get the edge lengths by raising the underlying exponential variables to the power $1/d$. To obtain a coupling, we introduce the \emph{$d$-PWIT}, which is just the ordinary PWIT modified by raising the edge lengths to power $1/d$. The original PWIT is the 1-PWIT, and the $(k, \theta)$-neighborhood of the root of the 1-PWIT corresponds to the $(k, \theta^{1/d})$-neighbor-hood of the root of the $d$-PWIT. 

By rescaling, the generalization of Lemma \ref{L:couplingGeneral} becomes:

\begin{Lemma} \label{L:reallyGeneral}
Let $k$, $\theta>0$, $n$ and $d\geq 1$ be given. Consider the pseudo-dimension $d$ mean field model on $n$ vertices, with $m$ vertices $v_1,\dots, v_m$ chosen independently of the edge-lengths. 

There is a coupling of this process to $m$ independent $d$-PWITs such that with probability at least $$1 - \frac{(m\theta^d+m+1)^k}{n^{1/3}},$$ the union of the $(k, \theta)$-neighborhoods of $v_1,\dots, v_m$ is isomorphic to the $(k, \theta)$-neighborhoods of the roots of the $d$-PWITs, with corresponding edges having equal length. \end{Lemma}

\subsection{Graph Exploration on the $d$-PWIT} \label{S:dPWIT}
Here we assume that $d\geq 1$, although some of the results hold also for $0<d<1$. In view of the results of Section \ref{S:PWITlike} it makes sense to study Graph Exploration played on the $d$-PWIT. If $v$ is a vertex of the $d$-PWIT we let $T_\theta(v)$ be the subgraph that can be reached from $v$ by downward paths consisting of edges of length at most $\theta$. The subgraph $T_\theta(root)$ is called the \emph{$\theta$-cluster}, and clearly nothing outside the $\theta$-cluster is relevant for the game. Notice that the underlying graph of the $\theta$-cluster is a Galton-Watson tree with Poisson$(\theta^d)$-distributed offspring.

Our main objective is to show that although a priori the game does not need to terminate, there is almost surely a unique sensible way of assigning to it a game-theoretical value. The precise statement is Proposition~\ref{P:uniquenessOfValuation} below, and this is the key to the proof of Theorem~\ref{T:main}. Proposition~\ref{P:uniquenessOfValuation} shows that as $\theta\to\infty$, a certain form of symmetry-breaking does not occur. This appears to be the fundamental reason why the replica and cavity methods are correct for the matching problem.

When we speak of the \emph{value} of a vertex $v$, by convention we mean the value of having moved to $v$, in other words the value of playing second if the game was played on $T_\theta(v)$ starting from $v$. If such a value $f(v)$ can be defined consistently, it must clearly satisfy
\begin{equation} \label{defValuation}
f(v) = \min(\theta/2, l_i - f(v_i)),
\end{equation}
where $l_i$ is the length of the edge to the $i$:th child $v_i$ of $v$, and the minimum is taken over $\theta/2$ and the sequence of $l_i-f(v_i)$ as $v_i$ ranges over all children. 

For a given realization of the $\theta$-cluster, we say that a function $f$ from its vertices to the real numbers is a \emph{valuation} if it satisfies \eqref{defValuation}. A valuation can be regarded as a consistent way for a player to assess the positions of the game. We observe the following: \begin{itemize} \item A valuation must satisfy $-\theta/2 \leq f(v) \leq \theta/2$ for every $v$. \item If $v$ is a leaf of the $\theta$-cluster, then $f(v) = \theta/2$. \item If the $\theta$-cluster is finite, there is a unique valuation. \end{itemize} 
 
\begin{Prop} \label{P:existenceOfValuation}
There is almost surely a valuation.
\end{Prop}

The only reason we say ``almost surely'' is that we haven't excluded the possibility that a vertex may have infinitely many children in the $\theta$-cluster. If this was the case, we would have to replace minimum by infimum in \eqref{defValuation}, but this is an event of zero probability.

\begin{proof}
Consider a ``partial valuation'' $f^k_B$ obtained by assigning values in favor of Bob to the vertices at distance $k$ from the root. More precisely, these vertices get value $\theta/2$ if $k$ is even and $-\theta/2$ if $k$ is odd. Values are then propagated towards the root according to \eqref{defValuation}. As $k$ increases, the values $f^k_B(v)$ form a monotone sequence at each vertex $v$ (decreasing at even levels, increasing at odd levels). Therefore there is a pointwise limit $$f_B(v) = \lim_{k\to\infty} f^k_B(v),$$ and it is easily verified that $f_B$ is a valuation.
\end{proof}

Clearly $f_B$ is at least as favorable to Bob as any other valuation. We can order the valuations from Bob's point of view by saying that $f_1 \leq f_2$ if whenever $v$ is at even distance from the root, $f_1(v) \leq f_2(v)$, and whenever $v$ is at odd distance from the root, $f_1(v) \geq f_2(v)$. Under this ordering the set of valuations forms a lattice where $f_B$ is the maximal element, and similarly there is a minimal element $f_A$ which is most favorable from Alice's point of view.

We are aiming to show that almost surely $f_A = f_B$. This holds trivially in the range $\theta\leq 1$, since the $\theta$-cluster is almost surely finite. For $\theta>1$, the $\theta$-cluster is infinite with positive probability, and the scenario that we wish to exclude is that at some critical value of $\theta$ there occurs a breaking of symmetry after which $f_A$ is distinct from $f_B$.

The question whether $f_A = f_B$ is in a curious way similar to questions of the efficiency of game-tree search in games of perfect information such as chess. Uniqueness of valuation means that a game-tree search will be effective, while symmetry-breaking corresponds to a situation where important long-term features of a position stay invisible to any fixed-depth search.

\subsection{The branching of near-optimal play} \label{S:branching}
For the moment we take $f_B$ as our default valuation. This defines a strategy in an obvious way: From a vertex $v$, terminate if $f_B(v)=\theta/2$, and otherwise move to the child $v_i$ for which $f_B(v) = l_i - f(v_i)$. There seems to be the possibility of a tie in which several move options would be consistent with $f_B$, but $f_B$ has the property that $f_B(v_i)$ depends only on $T_\theta(v_i)$. Therefore $l_i - f_B(v_i)$ has continuous distribution and is independent of $l_j - f_B(v_j)$ for $i\neq j$. It follows that the probability of a tie between move options is zero.

Let $\delta>0$. We say that a move from $v$ to $v_i$ is \emph{optimal} if $l_i - f_B(v_i) = f_B(v)$, and \emph{$\delta$-reasonable} if $l_i-f_B(v_i) \leq f_B(v)+\delta$. Let $R$ be the subtree of the $\theta$-cluster formed by all paths from the root consisting of $\delta$-reasonable moves by Alice and optimal moves by Bob (a move can be $\delta$-reasonable even if $l_i>\theta$ so some $\delta$-reasonable moves are excluded, but this is not important). Let $R(k)$ be the set of vertices of $R$ at distance $k$ from the root.

\begin{Prop} \label{P:finiteR}
If $\delta$ is sufficiently small, then $R$ is almost surely finite.
\end{Prop}

We let \begin{equation} \label{Hdef} H(k) = \#\left\{\text{$v \in R(k)$ : $f_B(v)<\theta/2$}\right\} + \frac12\cdot  \#\left\{\text{$v \in R(k)$ : $f_B(v)=\theta/2$}\right\} .\end{equation}
The proof of Proposition~\ref{P:finiteR} consists in showing that for sufficiently small $\delta$, $EH(k) \to 0$ as $k\to\infty$.

The event $v\in R(k)$ does not depend on $T_\theta(v)$ through anything else than $f_B(v)$. It follows that if we condition on $v\in R(k)$ and on $f_B(v)$, the structure of $T_\theta(v)$ is distributed as if we condition on $f_B(v)$ only. Therefore we first assume that $v$ is an ``arbitrary'' vertex of the $\theta$-cluster in the sense that $T_\theta(v)$ is itself equal in distribution to $T_\theta(root)$. The children of $v$ in the $\theta$-cluster are denoted by $v_i$.

\begin{Lemma} \label{L:Poisson}
The points $(l_i, f_B(v_i))$ constitute a two-dimensional inhomogeneous Poisson point process on the square $[0, \theta]\times [-\theta/2, \theta/2]$. \end{Lemma}

\begin{proof} The sequence of edge lengths $l_i$ is a Poisson point process. Since $f_B(v_i)$ depends only on $T_\theta(v_i)$, the $f_B(v_i)$'s are independent of each other and of the $l_i$'s. \end{proof}

By the \emph{$l$-$f$-square} we mean the square $[0, \theta]\times [-\theta/2, \theta/2]$. We let $\mu_v$ be the measure on the $l$-$f$-square associated with the Poisson process of pairs $(l_i, f_B(v_i))$.
The measure is degenerate on the line $f = \theta/2$ in the sense that this line has positive measure. Also notice that we do not assume that these point processes are equal in distribution for all $v$ (this is what we are about to prove). From what we have established so far it is conceivable that $\mu_v$ depends on whether $v$ is at even or odd distance from the root.

To bound $EH(k+1)$ in terms of $EH(k)$ we bound the expected number of moves in $R$ from a vertex $v\in R(k)$ in four cases, depending on whether Alice or Bob is about to move and conditioning either on $f_B(v) < \theta/2$ or on $f_B(v) = \theta/2$. The calculations rely crucially on the fact that for $d\geq 1$ the density of $l$ is increasing, and that therefore the measure $\mu_v$ of a subset of the $l$-$f$-square increases under translation to the right. 

We first consider the case that Alice is about to move from a vertex $v\in R(k)$, where thus $k$ is even. Suppose first that $f_B(v)<\theta/2$. Alice's optimal move is given by a point $(l_i, f_B(v_i))$ above the diagonal $l-f = \theta/2$ in the $l$-$f$-square. If we condition on $f_B(v) \in [a,b]$ for some $a, b$ such that $-\theta/2 \leq a\leq b < \theta/2$, then \begin{multline} P(f_B(v_i) = \theta/2) = \frac{\mu_v(f=\theta/2\,\&\, a+\theta/2 \leq l \leq b+\theta/2)}{\mu_v(l-f\in [a, b])} \\ \geq \frac{\mu_v(f=\theta/2\,\&\, a+\theta/2 \leq l \leq b+\theta/2)}{\mu_v(a+\theta/2 \leq l \leq b+\theta/2)} = \mu_v(f=\theta/2) \geq \exp(-\theta^d). \end{multline} It follows that the probability that $f_B(v_i)=\theta/2$ conditioning on $v\in R(k)$ and $f_B(v)<\theta/2$ is at least $\exp(-\theta^d)$, and that therefore the optimal move by Alice contributes to $EH(k+1)$ by at most $1-1/2\cdot\exp(-\theta^d)$.

The expected number of non-optimal $\delta$-reasonable moves is at most $\delta\cdot d\theta^{d-1} = o(1)$ as $\delta\to 0$. Hence the expected contribution to $H(k+1)$ when Alice moves from a vertex $v$ such that $f_B(v)<\theta/2$ is at most $$1-\frac12\exp(-\theta^d) + o(1).$$
By $o(1)$ we mean a term that can be made as small as we please by making $\delta$ small.

Consider now the case that Alice moves from a vertex $v\in R(k)$ with $f_B(v) = \theta/2$. Then there is no optimal move (the optimal decision is to terminate), and again the expected number of $\delta$-reasonable moves is at most $\delta\cdot d\theta^{d-1} = o(1)$. It follows that $$\frac{EH(k+1)}{EH(k)} \leq \max\left(1-\frac12\exp(-\theta^d) + o(1), \frac{o(1)}{1/2}\right) \leq 1-\frac12\exp(-\theta^d) + o(1).$$
When Bob moves, there is no optimal move if $f_B(v)=\theta/2$ and at most one if $f_B(v)<\theta/2$. Hence the growth factor for $H(k)$ over a pair of moves, one by Alice and one by Bob, satisfies $$\frac{EH(k+2)}{EH(k)} \leq 1-\frac12\exp(-\theta^d) + o(1) < 1,$$ uniformly in $k$ if $\delta$ is sufficiently small. It follows that $EH(k)\to 0$ as $k\to \infty$ and this completes the proof of Proposition~\ref{P:finiteR}.

The upper bounds on the expected contributions to $H(k+1)$ when moving from a vertex $v\in R(k)$ are summarized in the following table:

\medskip 

\begin{center}
  \begin{tabular}{@{} ccc @{}}
    \hline
    Player to move  & Vertex & Contribution to $EH(k+1)$\\ 
    \hline
    Alice moves & $f_B(v) < \theta/2$ & $1-1/2\cdot \exp(-\theta^d)+o(1)$\\ 
     & $f_B(v) = \theta/2$ & $o(1)$ \\ 
     \hline
    Bob moves & $f_B(v) < \theta/2$ & $1$ \\ 
     &  $f_B(v) = \theta/2$ & $0$\\ 
    \hline
  \end{tabular}
\end{center}

\medskip

\begin{Lemma} \label{L:termination}
For sufficiently small $\delta$, there is almost surely no infinite path starting anywhere in the $\theta$-cluster and consisting of optimal moves by Bob and $\delta$-reasonable moves by Alice. \end{Lemma}

\begin{proof}
If such a path started from the root, it would be a subset of $R$, and $R$ is almost surely finite. This event therefore has probability zero, and it follows that the probability of such a path anywhere in the $\theta$-cluster is also zero. \end{proof}

\begin{Prop} \label{P:uniquenessOfValuation}
There is almost surely only one valuation.
\end{Prop}

\begin{proof}
It suffices to show that almost surely $f_A(root) = f_B(root)$. Suppose therefore that this is not the case. Now let both Alice and Bob play ``optimistically'' in the sense that Alice plays according to $f_A$ and Bob plays according to $f_B$. Obviously they can never agree on an outcome of the game, so play has to continue forever. From Bob's perspective, it will seem that Alice sometimes makes mistakes that improve Bob's position. On the other hand the total gain (from Bob's perspective) of all these mistakes cannot be more than $\theta$, because the moment it adds up to more, Bob can terminate the game and receive a payoff greater than $f_B(root)$, and thereby also greater than $f_A(root)$, which is a contradiction. Therefore the game must eventually reach a point where Alice's all future mistakes relative to $f_B$ add up to at most $\delta$. The play from that point on will contradict Lemma~\ref{L:termination}.
\end{proof}

We need no longer distinguish between $f_A$ and $f_B$, and we denote the almost surely unique valuation by $f$. 
Now recall the partial valuations $f^k_B$, and define $f^k_A$ similarly by choosing the values at level $k$ in favor of Alice. Notice that $f^k_B$ and $f^k_A$ are the upper and lower bounds on $f$ that we get by looking $k$ moves ahead from the root.

\begin{Prop} \label{EdiffToZero}
$E\left(f^k_B(root) - f^k_A(root)\right)\to 0$ as $k\to\infty$.
\end{Prop}

\begin{proof}
We have established that almost surely there is only one valuation. This means that almost surely, $f^k_B(root) - f^k_A(root) \to 0$ monotonely as $k\to\infty$. The statement now follows from the principle of monotone convergence.
\end{proof}

\subsection{Interpretation in terms of integral equations} \label{S:integralEquation}
We want to obtain, to the extent possible, an analytical characterization of the distributions of $f_A^k(root)$, $f_B^k(root)$, and their common limit $f(root)$. We have \begin{equation} \label{recursion} f_A^{k+1}(root) = \min(\theta/2, l_i-f_A^{k+1}(v_i)),\end{equation} where $v_i$ ranges over the children of the root. Notice that $$f_A^{k+1}(v_i) \overset{\rm d}= f_B^k(root).$$ Clearly the same holds with the roles of Alice and Bob interchanged.

Suppose now that we describe the distribution of $f_B^k(root)$ by the function $$G_k(x) = P(f_B^k(root) \geq x),$$ and similarly $$F_{k+1}(x) = P(f_A^{k+1}(root) \geq x).$$ Then for $-\theta/2 \leq x \leq \theta/2$, $F_{k+1}(x)$ is the probability that there is no event in the inhomogeneous Poisson process of $v_i$ such that $l_i-f_A^{k+1}(v_i)<x$, or equivalently, that there is no $l_i$ such that $f_A^{k+1}(v_i) > l_i-x$. Here it doesn't matter whether the inequality is strict or not, so for given $x$ and $l_i$, $$P\left(f_A^{k+1}(v_i) > l_i-x\right) = G_k(l_i-x).$$ The sequence of $l_i$ such that $f_A^{k+1}(v_i) > l_i-x$ is therefore the set of points in a thinned Poisson point process of rate $dl^{d-1}G_k(l-x)$, and it follows that $$F_{k+1}(x) = \exp\left(-d\int_0^{\theta/2+x} l^{d-1}G_k(l-x)\dd l\right).$$ Therefore we define an operator $V_\theta$ on functions on the interval $[-\theta/2, \theta/2]$ by $$(V_\theta F)(x) = \exp\left(-d\int_0^{\theta/2+x}l^{d-1}F(l-x)\dd l\right).$$ We have $F_{k+1} = V_\theta(G_k)$, and by reversing the roles of Alice and Bob, $G_{k+1} = V_\theta(F_k)$. The distributions of $f_A^k(root)$ and $f_B^k(root)$ are thus obtained by starting from $F_0 = 0$ and $G_0 = 1$ (on the interval $[-\theta/2, \theta/2]$) and iterating the operator $V_\theta$. But since $G_1 = G_0$ it follows inductively that $F_2 = F_1$, $G_3=G_2$ and so on. Therefore in reality there is only one sequence of functions, obtained by iterating $V_\theta$ starting from the zero function.

The operator $V_\theta$ is decreasing in the sense that if $F(x)\leq G(x)$ for every $x$, then $(V_\theta F)(x) \geq (V_\theta G)(x)$ for every $x$. It follows that if we start from the function which is identically zero (or identically 1) and iterate, the sequence of functions must either converge to a fixed point or approach an attractor of period 2. Proposition \ref{P:uniquenessOfValuation} is equivalent to the statement that for every $\theta>0$ and every $d\geq 1$, the sequence converges to a fixed point. Actually it is easy to see that if we start from any real integrable function $F$, then after two iterations we have a function which takes values in $[0,1]$, in other words lies between $F_0$ and $G_0$. Therefore the subsequent iterates will be squeezed between $F_k$ and $G_k$ and thus converge to the same fixed point. In particular $V_\theta$ has only one fixed point. 

The similarity to the M\'ezard-Parisi integral equation \eqref{integralEquation} is clearly visible. Naturally we may define an operator $V_\infty$ by $$(V_\infty F)(x) = \exp\left(-d\int_0^{\infty}l^{d-1}F(l-x)\dd l\right).$$

It seems clear both from numerical evidence and in view of the results we have established, that as $\theta\to \infty$, the fixed point of $V_\theta$ should converge uniformly to a limit function which is a unique fixed point to $V_\infty$, in other words a unique solution to the M\'ezard-Parisi equation \eqref{integralEquation}. We certainly believe that a more detailed analysis will show this to be true (possibly the ideas of \cite{SS09} can be extended to $d>1$), but we leave it as an open conjecture since it is not necessary for our proof of Theorem~\ref{T:main}. Moreover, the natural way to obtain numerical results from \eqref{integralEquation} is to approximate $F(x)$ by 1 for large negative $x$ and by 0 for large positive $x$. Therefore in practice the numerical results based on \eqref{integralEquation} reduce to to solving the equation $V_\theta(F) = F$ on a bounded interval. 

\subsection{The density of the minimum diluted matching}
We now return to the mean field model $K_n$ on $n$ vertices. Suppose that $\theta$ and $d\geq 1$ are fixed and let the random variable $q_n$ be the proportion of vertices that are not matched (for which we pay the punishment of $\theta/2$) in the optimum diluted matching. Here and in the following we let $q = P(f = \theta/2) = F(\theta/2)$, where $F$ is the fixed point of $V_\theta$. In other words $q$ is the probability that Alice quits immediately in Graph Exploration on the $d$-PWIT. 

\begin{Prop} \label{P:density} As $n\to\infty$, $q_n \overset {\rm p}\to q$.
\end{Prop}

\begin{proof} We show that $Eq_n \to q$ and ${\rm var}(q_n) \to 0$. Let $k$ be a positive integer. With probability $1-o(1)$ as $n\to\infty$, the $(k,\theta)$-neighborhood of a given vertex $v$ in $K_n$ is isomorphic to the first $k$ levels of a $d$-PWIT. 

By choosing large $k$, we can make $E\left(f^k_B(root) - f^k_A(root)\right)$ as small as we please, and provided that the coupling to the $d$-PWIT succeeds, the game theoretical value of Graph Exploration on $K_n$ starting at $v$ is between $f^k_A(root)$ and $f^k_B(root)$.
Therefore conditioning on success of the coupling to the $d$-PWIT, $$P(f^k_A(root) = \theta/2) \leq P(\text{$v$ is not matched}) \leq P(f^k_B(root) = \theta/2),$$ and both sides converge to $q$ as $k\to\infty$.

To bound the variance of $q_n$ we simply take two vertices $v_1$ and $v_2$ of $K_n$ and estimate the probability that neither is matched. To do this we apply Lemma \ref{L:reallyGeneral} with $m=2$. With high probability the $(k, \theta)$-neighborhoods of $v_1$ and $v_2$ are disjoint and isomorphic to the $(k, \theta)$-neighborhoods of the roots of two independent $d$-PWITs. It follows that the probability that neither is matched converges to $q^2$.
\end{proof}

\subsection{The cost of the minimum diluted matching} \label{S:dilutedCost}
We wish to find the normalized limit cost of the minimum diluted matching. This cost splits naturally into the length of the participating edges and the cost of the penalties for the unmatched vertices. The penalties have been taken care of in the previous section, and therefore we concentrate on the participating edges. We let $M_n(\theta)$ be the total length of the participating edges in the optimum diluted matching.

\begin{Thm} \label{T:thetaMain}
For each $\theta$ and $d\geq 1$, there is a number $\beta_\theta(d)$ such that \begin{equation} \label{thetaMainEq} \frac{M_n(\theta)}{n} \overset{\rm p}\to \beta_\theta(d). \end{equation}
\end{Thm}

\begin{proof}
Recall that the edge lengths are distributed like $(nX)^{1/d}$, where $X$ is exponential of mean $n$. Therefore the density function of the length of a single edge is $$\frac{dl^{d-1}}{n}\cdot\exp\left(-\frac{l^d}{n}\right).$$ The expectation of $M_n(\theta)$ is the total number of edges in the graph times the expected contribution to $M_n(\theta)$ from a single edge:
 \begin{equation} \label{contribution} EM_n(\theta) = \binom{n}2\cdot \frac{d}n\cdot\int_0^{\theta} l^d\cdot \exp(-l^d/n)\cdot P(\text{participation given length $l$})\dd l.\end{equation} 
Deleting the factor $\exp(-l^d/n)$ will introduce an error of at most a factor $(1-\theta^d/n)$. Normalizing to obtain a quantity of order 1, we get \begin{equation} \label{middlestage} \frac{EM_n(\theta)}{n} = \frac{d}2\cdot\int_0^{\theta} l^d\cdot P(\text{participation given length $l$})\dd l + o(1).\end{equation}

We now choose a positive integer $k$. We explore the $(k, \theta)$-neighborhood of the endpoints $u$ and $v$ of the edge $e$ and discard the cases of ``failure'' when we cannot successfully couple to two independent $d$-PWITs to depth $k$. By choosing $k$ suitably as a function of $n$ we can make $k$ tend to infinity while the probability of failure is $o(1)$. 

Given that the coupling succeeds, the maximum length at which $e$ participates lies between $f_A^k(u') + f_A^k(v')$ and $f_B^k(u') + f_B^k(v')$, where $u'$ and $v'$ are the roots of the two $d$-PWITs. Replacing these bounds by $f(u') + f(v')$ will introduce another error of $o(1)$. 
Hence \eqref{middlestage} is equal to $$\frac{d}2\cdot \int_0^{\theta} l^d\cdot P\left(l\leq f(u') + f(v')\right)\dd l + o(1).$$ Here $f(u')$ and $f(v')$ are independent and satisfy $P(f\geq x) = F(x)$ where $F$ is the fixed point of $V_\theta$.
By partial integration it follows that \begin{equation} \label{doubleIntegral} \frac{EM_n(\theta)}{n} \to \frac{d^2}2\cdot \underset{\substack{ -\theta/2<x,y<\theta/2 \\ x+y\geq 0\\}}{\int\int} (x+y)^{d-1}F(x)F(y)\dd x\dd y.\end{equation}

We denote the right hand side of \eqref{doubleIntegral} by $\beta_\theta(d)$. To see that \eqref{doubleIntegral} can be strengthened to convergence in probability as stated in \eqref{thetaMainEq} we again apply Lemma~\ref{L:reallyGeneral}, this time with $m=4$. It follows that the expected contribution from an arbitrary pair of edges to the square of $M_n(\theta)$ is asymptotically the same as the square of the expected contribution of one edge, and that therefore ${\rm var}(M_n(\theta)) = o(n^2)$. \end{proof}

\subsection{Perfect matching} \label{S:perfectMatching}
Here we complete the proof of Theorem~\ref{T:main}. Assuming that $n$ is even, we study the length $M_n$ of the minimum perfect matching. Naturally we expect perfect matching to correspond to infinite $\theta$, and the remaining step essentially amounts to showing that we can interchange the order in which $n$ and $\theta$ go to infinity. For the bipartite graph and $d=1$ this was proved in \cite{A92}. Without claims of originality we give a self-contained proof valid for $d>0$. This proof is based on expander properties of random graphs along the same lines as \cite{F04}. I thank David Aldous for pointing out that the method of \cite{F04} applies here.  

Recall that $\beta_\theta(d)$ is defined as the right hand side of \eqref{doubleIntegral}. Clearly $\beta_\theta(d)$ is upper-bounded according to \eqref{upperBound} and increasing in $\theta$ by \eqref{thetaMainEq}. We define $$\beta_M(d) = \lim_{\theta\to\infty} \beta_\theta(d).$$ 
What remains is to show that for every $\epsilon>0$, $$P\left(\frac{M_n}{n}\leq \beta_M(d)+\epsilon\right) \to 1$$ as $n\to\infty$.

For every $\epsilon>0$ and every $q>0$, we can find an $n$ and a $\theta$ such that with as high probability as we please, the minimum diluted matching has cost at most $(\beta_M(d) + \epsilon)\cdot n$ and density at least $1-q$ in terms of vertices covered. 
Therefore in order to complete the proof of Theorem~\ref{T:main} it suffices to show the following:

\begin{Prop} \label{P:completing} With high probability, a partial matching that covers a $(1-q)$-fraction of the vertices can be completed to a perfect matching in a way that increases the total length by at most $\delta n$, where $\delta$ depends on $q$ but not on $n$, and $\delta\to 0$ as $q\to 0$. \end{Prop}

We introduce some extra edges by letting each pair of vertices give rise to a Poisson process of edges. More precisely, for each pair of vertices there is a rate 1 Poisson point process on the positive real numbers, and we let the sequence of edges have lengths given by $(nX_i)^{1/d}$ where $X_i$ are the points of the process. Obviously the extra edges do not change the minimum matching. We randomly color every edge red or green, where the probability of red is $1-p$ and the probability of green is $p$, for some $p$ which will go to zero as $n$ goes to infinity. 

First we find the minimum diluted matching $M_{Red}$ on the red edges. Then before looking at the green edges we choose arbitrarily a bipartition of the vertices into two sets $A$ and $B$ of size $m=n/2$ such that every edge of $M_{Red}$ connects a vertex of $A$ to a vertex of $B$. Then we look at the green edges that connect $A$ to $B$ and give each of them a random orientation by independent coin flips. We let $D$ be the set consisting of the 13 cheapest green edges directed from each vertex to the opposite side of the partition. 

\begin{Lemma} With high probability $D$ has the following expander property: If $S$ is a set of vertices from one side of the partition, and $1\leq \left|S\right| \leq m/3$, then $\left|S'\right| > 2\left|S\right|$, where $S'$ denotes the set of $D$-neighbors of $S$. \end{Lemma}

\begin{proof} If this condition is violated, then there is a positive integer $s\leq m/3$ and a set of $s$ vertices on one side of the partition such that all its $13s$ edges go into a certain set of $2s$ vertices on the other side. The probability that this happens is at most $$2\cdot\sum_{1\leq s\leq m/3} \binom{m}{s}\binom{m}{2s}\left(\frac{2s}{m}\right)^{13s}.$$ Using the standard inequality $$\binom{m}{k} \leq \left(\frac{me}{k}\right)^k,$$ we find that the failure probability is at most $$2\cdot\sum_{1\leq s\leq m/3} e^{3s}2^{11s}\left(\frac{s}{m}\right)^{10s}.$$ 

By log-convexity of the summand, the maximum of the terms is attained by either the first or the last one. The first term ($s=1$) is $O(1/n^{10})$ and the last one is at most $$\left(\frac{e^32^{11}}{3^{10}}\right)^{m/3} \leq 0.95^{n}.$$ For large $n$ the first term will dominate, and since there are $O(n)$ terms, the failure probability of the expander property is $O(1/n^9)$. 
\end{proof}

The number 13 was chosen to make the constant in the left hand side smaller than 1. If the expander property does not hold, then our scheme fails. If we want to keep the expected cost of completing the matching to $o(n)$ we may have to replace the number 13 by a larger number depending on $d$. We can then, in the cases of failure, just pick an arbitrary matching of red edges and the expected length will be $O(n^{1+1/d})$. This will do provided that the probability of failure is sufficiently small.

\begin{proof}[Proof of Proposition~\ref{P:completing}] If the expander property holds, then we complete $M_{Red}$ to a perfect matching by using the edges in $D$. We extend the matching successively by finding an alternating path that connects two unmatched vertices. When $k$ unmatched vertices remain on each side, we can find such an alternating path of length $O(\log(n/k))$. A simple calculation shows that the total number of edges of $D$ that become involved in completing the matching is $O(nq\log(n/q))$. Notice that this bound is deterministic and holds whenever the expander property holds.

The completion of the matching is done independently of the actual lengths of the edges in $D$, so the expected total cost of completing the matching (given that the expander property holds) is bounded by $O(nq\log(n/q))$ times the expected length of an edge in $D$. Now it is clear that if $p\to 0$ as $n\to\infty$, we can keep the total length of $M_{Red}$ to $M_n(\theta)+o(n)$ while at the same time keeping the cost of the completion process to $o(n)$.   
\end{proof}

This completes the proof of Theorem~\ref{T:main}.

\subsection{The $\pi^2/12$-limit for $d=1$}
The case $d=1$ corresponds to the model studied by Aldous in \cite{A92, A01}. In our terminology his result (anticipated in \cite{MP85}) is that $\beta_M(1) = \pi^2/12$. We briefly show how to derive this from our present approach. In \cite{A01, MP85} the calculations start from equation \eqref{integralEquation} which has the solution $F(x) = 1/(1+e^x)$. This corresponds to infinite $\theta$, but in our approach we arrive at the equation $V_\theta(F) = F$ for finite $\theta$, which for $d=1$ becomes $$F(x) = \exp\left(-\int_0^{\theta/2+x}F(l-x)\dd l\right).$$
Since we know that the equation has a unique solution, it suffices to verify that $$F(x) = \frac{1+q}{1+e^{(1+q)x}}$$ solves it, where $q$ satisfies $$\theta = \frac{-2\log q}{1+q}.$$ Via \eqref{doubleIntegral} it can then be verified that the limit cost for finite $\theta$ is given by \begin{equation} \label{explicitbeta} \beta_\theta(1) =  \int_q^1\frac{-\log t}{1+t}\\d t,\end{equation} from which the $\pi^2/12$-result is obtained by putting $q=0$.
 The limit cost \eqref{explicitbeta} for the minimum density $1-q$ matching also follows in a completely different way from the results of \cite{W09}. A more streamlined derivation along the present lines is given in work in preparation jointly with G.~Parisi.

\section{The traveling salesman problem} \label{S:TSP}
\subsection{Analog of Theorem~\ref{T:main} for the TSP}
In Sections~\ref{S:generalizedGraphExploration}--\ref{S:limitTSP} we establish the analog of Theorem~\ref{T:main} for the traveling salesman problem. It follows from a theorem of A.~Frieze \cite{F04} that in the case $d=1$, the length of the traveling salesman tour is asymptotically the same as the length of the polynomially solvable 2-factor problem. The idea is to ``patch'' the minimum 2-factor to a tour by replacing $o(n)$ edges, and to show that this can be done at small increase in total length. The proof is similar to our proof of Proposition~\ref{P:completing} but more complicated due to the global constraint of the TSP. By the concavity of the function $X\mapsto X^{1/d}$, the theorem extends automatically to $d\geq 1$ but in fact, under minor changes, Frieze's proof works also for $0<d<1$. We do not discuss the details of Frieze's result here, but it means that we can obtain results for the TSP by studying the more tractable 2-factor problem. 

Our treatment of the 2-factor/TSP closely parallels the matching problem. We focus on the differences, and omit the details where they are similar to those of the matching problem. 

Let $L_n$ be the length of the minimum traveling salesman tour. The analog of Theorem~\ref{T:main} is
\begin{Thm} \label{T:mainTSP} For every $d\geq 1$ there is a number $\beta_{TSP}(d)$ such that 
\begin{equation} \label{mainConjectureTSP} \frac{L_n}{n} \overset{\rm p}\to \beta_{TSP}(d).\end{equation}
\end{Thm}
Since we divide the total length by $n$, $\beta_{TSP}(d)$ is the limit average length of an edge in the optimum solution. It was proved in \cite{W09} that $\beta_{TSP}(1) \approx 2.0415481864$ can be expressed as $$\frac12\int_0^\infty y\dd x,$$ where $y$ satisfies $$\left(1+\frac x2\right)e^{-x} + \left(1+\frac y2\right)e^{-y} = 1.$$
We have computed $\beta_{TSP}(2)$ numerically, and found that $$ \beta_{TSP}(2) \approx 1.285153753372032.$$ This is consistent with the value 0.7251 given in \cite{KM89, PM98}, since the normalizations differ by a factor $\sqrt{\pi}$.

It is interesting to compare this value to the famous Beardwood-Halton-Hammersley constant \cite{BHH59} for the euclidean TSP in two dimensions. In \cite{CBBMP97} this constant was estimated to $0.7120$, but obtaining rigorous numerical bounds has proved annoyingly difficult. With the ``euclidean'' normalization (dividing by $\sqrt{\pi}$), our value for $\beta_{TSP}(2)$ is $0.725070360909803$, which is within $2\%$ of its euclidean counterpart. It is worth pointing out that in the euclidean setting, the TSP does not seem to be equivalent to the 2-factor problem. 

\subsection{Generalized Graph Exploration} \label{S:generalizedGraphExploration}
To carry out the analysis for the 2-factor problem, we generalize the game of Graph Exploration to a setting where each vertex has a nonnegative integer \emph{capacity}. The capacity is a bound on the degree of the vertex in a feasible solution to the corresponding optimization problem. The matching problem is obtained by setting all capacities to 1, and in the 2-factor problem all capacities are equal to 2. In the generalized game, the rules are as follows:

\begin{itemize}

\item The game starts at a specific vertex, and the players take turns choosing the edges of a walk (not necessarily self-avoiding).

\item A player who chooses an edge pays the length of that edge to the opponent.

\item Each player can use an edge at most once, and the set of edges chosen by a player must satisfy the capacity constraints, that is, each player can use each vertex at most a number of times equal to its capacity.

\item Moreover, Bob can use the starting point only one time less than its capacity. One way of thinking about this is to regard the game as starting by Bob entering the graph at the starting point through an edge coming from the outside, and that therefore Bob has already used the starting point once.   

\item A player can, at any time, terminate the game by paying $\theta/2$ to the opponent.

\end{itemize} 

\subsection{The diluted flow problem}
The generalization of the matching problem to arbitrary capacities was called a \emph{flow problem} in \cite{W09}. In the \emph{diluted flow problem} there is a parameter $\theta$, and a feasible solution is a set of edges such that no vertex exceeds its capacity. The cost of a solution is the sum of the lengths of its edges and a penalty for vertices that are not used up to their capacity. The penalty for a vertex is $\theta/2$ times the difference between its capacity and the number of edges incident to it in the solution. 

In this setting, the proof of Proposition \ref{P:payoff} goes through almost word by word, provided that the interpretation of $G-v$ is that the capacity of $v$ has been decreased by 1, and of course that $M(G)$ is replaced by the cost of the diluted flow problem.

\subsection{Generalized Graph Exploration on the $d$-PWIT}
We are led to study Generalized Graph Exploration on the $d$-PWIT. An important difference compared to the capacity 1 case is that when the capacities are greater than 1, it is possible to move upwards (towards the root) in the PWIT. Since the PWIT is a tree, a move upwards implies that the opponent has already used the edge in a downward move, so that after an upward move it is no longer possible to go back to that subtree. Moreover, since an upward move means that the player pays back what the opponent had paid to go downwards through the same edge, an upward move has the effect of canceling the opponent's downward move. In fact an upward move means that all moves played by one player in the subtree are canceled by moves of the opponent along the same edge in the opposite direction.

This leads to an alternative formulation of the game: When the game reaches a vertex $v$, the player who is not making the next move has the right to forbid a number of move options equal to the capacity of $v$ minus 1. In particular, when all capacities are equal to 2, it means that the player not to move can forbid one move option. With the alternative formulation, the moves that are actually carried out constitute a downward path in the tree. If the graph is not a tree, the alternative formulation becomes slightly more complicated: The effect of reversing the opponent's last move (in the original version of the game) is not quite the same as canceling it, since the players have used up one of their potential visits to a vertex that may be visited later on.

Since the PWIT is a tree, the alternative formulation is correct. The game starts by Bob forbidding one of Alice's move options. Then Alice makes a move and from that point forbids one of Bob's move options, and so on. The similarity to the chess variants \emph{refusal chess} and \emph{compromise chess} \cite{chessVariants} is obvious. Games where a player can forbid some of the opponent's move options have been called \emph{comply-constrain games} in the literature on combinatorial games \cite{SS02, GS04}. I was aware of the work on these games but still found their connection to the TSP quite surprising.  

A valuation is now redefined as a function on the vertices of the $\theta$-cluster that satisfies
$$f(v)  = \min(\theta/2, {\min}_2(l_i - f(v_i))),$$
where $\min_2$ denotes second-smallest. With the new definition, the set of valuations is still a lattice with a maximal element $f_B$ and a minimal element $f_A$. The crucial step in the proof of Theorem~\ref{T:mainTSP} is to show that again $f_A$ and $f_B$ are almost surely equal.

In the new setting, Proposition \ref{P:existenceOfValuation} and Lemma \ref{L:Poisson} hold with obvious modifications. We let $\mu_v$ be the measure associated to the process of $(l_i, f_B(v_i))$ with the new definition of $f_B$.

\subsection{The branching of near-optimal play}
Again we let $R$ be the set of paths in the $\theta$-cluster where Bob plays optimally (relative to $f_B$) and every action made by Alice is $\delta$-reasonable. A difference is that now there are two different types of actions. When Alice is about to make a move, Bob forbids the move option that would be most preferable to Alice, and Alice chooses between the remaining ones. The other type of action is that when Bob is about to move, Alice has to forbid one of Bob's move options (if there are any in the $\theta$-cluster). If Bob is about to move from $v$, and $v_i$ and $v_j$ are the best and second-best move options relative to $f_B$, then allowing Bob to play to $v_i$ is $\delta$-reasonable if $l_j-f_B(v_j) \leq l_i-f_B(v_i) + \delta$. Since Bob plays optimally we do not have to distinguish between Alice forbidding different other move options. 

We define $R(k)$ as before, but we modify the definition of $H(k)$ by introducing another parameter $\lambda$, assuming $0<\lambda<1$. We now let \begin{equation} H(k) = \#\left\{\text{$v \in R(k)$ : $f_B(v)<\theta/2$}\right\} + \lambda \cdot  \#\left\{\text{$v \in R(k)$ : $f_B(v)=\theta/2$}\right\} .\end{equation}

When Alice is moving, the best move option has been forbidden by Bob, but the analysis of Section~\ref{S:branching} still goes through, and we arrive at \begin{multline} \frac{EH(k+1)}{EH(k)} \leq \max\left(1-(1-\lambda)\exp(-\theta^d) + o(1), \frac{o(1)}{\lambda}\right) \\ \leq 1-(1-\lambda)\exp(-\theta^d) + o(1).\end{multline}

We now turn to the situation when Bob is about to move. The only type of mistake is now if Alice forbids a move option other than the best one, thereby allowing Bob to play a ``super-optimal'' move. We begin with the case that Bob moves from a vertex $v$ with $f_B(v) = \theta/2$. Then there is no optimal (that is, second-best) move, and the number of super-optimal moves is either zero or one. To upper-bound the branching we may condition on exactly one point above the diagonal $l-f<\theta/2$ in the $l$-$f$-square. The probability that neglecting to forbid the move corresponding to this point is a $\delta$-reasonable decision is $$\frac{\mu_v(\theta/2-\delta\leq l-f \leq \theta/2)}{\mu_v(l-f\leq \theta/2)} \leq \frac{\delta\cdot d\cdot \theta^{d-1}}{\exp(-\theta^d)} = o(1).$$ 

Finally we consider the case that Bob moves from a vertex $v$ such that $f_B(v)<\theta/2$. This means that there are at least two points of the process $(l_i, f_B(v_i))$ above the diagonal in the $l$-$f$-square. 
We allow for an optimal (that is, second-best) move and a super-optimal move to a vertex $v_i$ with $f_B(v_i) = \theta/2$. On the other hand we have to bound the probability of a super-optimal move to a vertex $v_i$ with $f_B(v_i) < \theta/2$.

Since $\mu_v$ is continuous except on the line $f=\theta/2$, we can find an $x<\theta/2$ such that $$\mu_v(x\leq f < \theta/2)$$ is as small as we please.

If $f_B(v)\leq -x$, then the probability of a super-optimal move to a vertex $v_i$ with $f_B(v_i)<\theta/2$ is at most \begin{multline} \frac{\mu_v(l-f\leq -x\, \& \, f<\theta/2)}{\mu_v(l-f\leq -x)} \leq  \frac{\mu_v(x\leq f<\theta/2 \, \& \, l\leq \theta/2-x)}{\mu_v(x\leq f\, \& \, l\leq \theta/2-x)} \\ \leq \frac{\mu_v(x\leq f<\theta/2)}{\mu_v(f = \theta/2)} \leq \frac{\mu_v(x\leq f<\theta/2)}{\exp(-\theta^d)},
\end{multline}
which can be made as small as we please.

Suppose on the other hand that $f_B(v)>-x$. Then the probability of a super-optimal move to a $v_i$ with $f_B(v_i)<\theta/2$ is at most $$\frac{\mu_v(f_B(v)-\delta\leq l-f < f_B(v))}{\mu_v(l-f < f_B(v))} \leq \frac{\delta\cdot d\cdot \theta^{d-1}}{\mu_v(l-f\leq -x)}.$$

By first choosing $x$ and then choosing $\delta$ we can make this too as small as we please. Therefore we can summarize the bounds on the expected contributions to $H(k+1)$ from the various move situations in the following table:

\begin{center}
  \begin{tabular}{@{} ccc @{}}
    \hline
    Player to move  & Vertex & Contribution to $EH(k+1)$\\ 
    \hline
    Alice moves & $f_B(v) < \theta/2$ & $1-(1-\lambda)\exp(-\theta^d)+o(1)$\\ 
     & $f_B(v) = \theta/2$ & $o(1)$ \\ 
     \hline
    Bob moves & $f_B(v) < \theta/2$ & $1+\lambda+o(1)$ \\ 
     &  $f_B(v) = \theta/2$ & $o(1)$\\ 
    \hline
  \end{tabular}
\end{center}
     
If $k$ is even, then Alice moves from vertices at level $k$, and we obtain a recursive bound on $EH(k+1)$ by
\begin{multline} \label{qwe1} \frac{EH(k+1)}{EH(k)} \leq \max\left(1-(1-\lambda)\exp(-\theta^d) + o(1), \frac{o(1)}\lambda\right) \\= 1-(1-\lambda)\exp(-\theta^d)+o(1).\end{multline}
When Bob moves from level $k+1$, we similarly obtain 
\begin{equation} \label{qwe2} \frac{EH(k+2)}{EH(k+1)} \leq \max\left(1+\lambda + o(1), \frac{o(1)}\lambda\right) \\= 1+\lambda+o(1).\end{equation}
Multiplying \eqref{qwe1} and \eqref{qwe2} we obtain a bound on the branching effect of a pair of moves, one by Alice and one by Bob:
$$\frac{EH(k+2)}{EH(k)} \leq 1-\exp(-\theta^d)+\lambda+\lambda^2\exp(-\theta^d)+o(1) < 1,$$ for small $\lambda$. It follows that if first $\lambda$ and then $\delta$ are chosen small enough but positive, then $EH(k)\to 0$ as $k\to \infty$, and consequently $R$ is almost surely finite.

In the same way as in Section~\ref{S:branching} it follows that there is almost surely only one valuation, and that $E\left(f^k_B(root) - f^k_A(root)\right)\to 0$ as $k\to\infty$.

\subsection{Integral equation for the TSP}
Let $f$ be the unique valuation. We can state an integral equation that describes the distribution of $f(root)$. Let $$F(x) = P(f(root)\geq x) = P(\text{at most one $i$ such that $l_i-f(v_i)\leq x$}),$$ assuming that $-\theta/2<x<\theta/2$. The process of $l_i$ such that $f(v_i)>l_i-x$ is an inhomogeneous Poisson point process of rate $dl^{d-1}F(l-x)$ on the interval $0\leq l \leq \theta/2+x$. Since $F(x)$ is the probability of at most one event in this process, $F$ must satisfy $$F(x) = (1+I(x))e^{-I(x)},$$ where $$I(x) = d\int_0^{\theta/2+x}l^{d-1}F(l-x)\dd l.$$ It is natural to define an operator $W_\theta$ by \begin{multline} \notag (W_\theta G)(x) \\= \left(1+d\int_0^{\theta/2+x}l^{d-1}G(l-x)\dd l\right)\cdot\exp\left(-d\int_0^{\theta/2+x}l^{d-1}G(l-x)\dd l\right).\end{multline} If we start with the function which is identically zero and iterate $W_\theta$, we obtain the distributions of $f_A^k(root)$ and $f_B^k(root)$ for successive values of $k$. Almost sure uniqueness of valuation is equivalent to the statement that the sequence converges pointwise, and this in turn implies that $F$ is the unique fixed point of $W_\theta$.

\subsection{The limit average length of an edge in the minimum tour} \label{S:limitTSP}
In complete analogy with the case of matching, the total length $L_n(\theta)$ of the edges in the optimum diluted 2-factor satisfies \begin{equation} \label{anEq} \frac{L_n(\theta)}{n} \overset{\rm p}\to \frac{d^2}2\cdot \underset{\substack{ -\theta/2<x,y<\theta/2 \\ x+y\geq 0\\}}{\int\int} (x+y)^{d-1}F(x)F(y)\dd x\dd y,\end{equation} where $F$ is now the fixed point of $W_\theta$ instead of $V_\theta$. The argument of Section~\ref{S:dilutedCost} carries over without changes. 

The quantity in the right hand side of  \eqref{anEq} is increasing in $\theta$, and converges to a limit $\beta_{TSP}(d)$ as $\theta\to \infty$. To show that $\beta_{TSP}(d)$ is the limit average length of an edge in the minimum traveling salesman tour, we have to do two things: First we show that $\beta_{TSP}(d)$ is the limit average length of an edge in the minimum 2-factor. The argument given in Section~\ref{S:perfectMatching} goes through without essential changes. We only need to modify the proof of Proposition~\ref{P:completing} by introducing extra edges of another color, say blue, to take into account the fact that a vertex may need two more edges to complete the 2-factor (there are also other ways to modify the proof). At the same time, we may remove the assumption that $n$ is even (this is another minor point that can be handled in several ways). Second, we apply the theorem of Frieze \cite{F04} to conclude that $\beta_{TSP}(d)$ is also the limit average edge-length in the minimum tour. This establishes Theorem~\ref{T:mainTSP}.

\section{The minimum edge cover} \label{S:edgeCover}
\subsection{Another optimization problem}
Finally we turn to another optimization problem belonging to the same family (minimizing the total length of an edge set satisfying certain constraints). An \emph{edge cover} is a set of edges of which every vertex  is incident to at least one. The minimum edge cover has been studied in the bipartite pseudo-dimension 1 model in \cite{HW09}, using ``finite $n$'' combinatorial methods. The limit cost is the area of the union of the regions $y\leq e^{-x}$ and $x\leq e^{-y}$ in the positive quadrant. This area is equal to $W(1)^2+2W(1)\approx 1.456$, where $W$ is the Lambert $W$-function defined as the inverse to the function $We^W$. In particular $W(1)\approx 0.567$ is the solution to $x=e^{-x}$, and gives the coordinates of the point of intersection of the two curves $y = e^{-x}$ and $x = e^{-y}$.

The edge cover problem has not been studied in the physics literature, but it has several interesting features and the replica-cavity method produces a prediction of the ground state energy with the same scheme of calculations as for matching and TSP. In contrast to matching and TSP, the number of edges in the optimum edge cover is not determined by the number $n$ of vertices, but depends on the problem instance. Despite this, the edge cover problem is technically simpler than matching and TSP, and in particular allows for an explicit solution (up to numerical constants) also for $d=2$.

It turns out that the edge cover problem corresponds to a two-person game and that using this game, the steps of the solution can be carried out in analogy with matching and TSP. We try to explain how to discover the correct definition of this game starting from the optimization problem. 
Eventually we introduce a diluted form, but for the moment we need not worry about the parameter $\theta$. 

We ask under what conditions an edge $e$ between vertices $u$ and $v$ belongs to the minimum edge cover. For large $n$ we can assume that $e$ is not part of any short cycle, so in practice we think of $e$ as a bridge whose removal will disconnect the graph into two subgraphs. Let $G$ be the subgraph containing $v$. Alice and Bob are now advocating the options of leaving out $e$ and of including $e$ respectively, focusing on $G$. 

Alice, who doesn't use the edge $e$, simply has to find an edge cover on $G$. Bob on the other hand uses $e$ and therefore has to find a set of edges that covers all vertices except possibly $v$. In contrast to the matching problem, Bob's task is not equivalent to the edge cover problem on the subgraph $G-v$, since he can use edges incident to $v$. 

The game therefore starts by Alice explaining how she intends to satisfy the constraint that Bob does not have, namely how to cover $v$. Possibly Alice uses more than one edge from $v$ in her cover, but she only needs to specify one of them in order to get rid of the constraint of having to cover $v$. Suppose Alice chooses to use the edge $(v, w)$. Then apart from covering $v$ she has fulfilled the constraint to cover $w$. That means that Bob is now behind in the sense that he has a constraint that Alice no longer has, and he therefore has to explain how he is going to cover $w$. 

Bob therefore chooses an edge from $w$, and there is nothing that prevents him from choosing the edge $(v, w)$ that Alice is using (but the initial assumption is that Bob is using the edge $(u,v)$, and this is the reason why in the game corresponding to matching, Bob does not have the right to go back to the starting point). If Bob chooses $(v,w)$, the game reaches an equilibrium where the players' remaining constraints are exactly the same, which means that the game is over. Since the length of the edge $(v,w)$ has been paid back and forth between the two players, the outcome of the game is zero. 

If on the other hand Bob chooses another edge $(w, x)$, then again Alice is behind because she has not yet covered $x$ and Bob has. Alice then chooses an edge from $x$, and if that edge goes to $v$ or $w$, equilibrium is reached and the game is over. Otherwise Alice covers a new vertex which Bob must then cover, and so on.

To sum up, an ``infinite $\theta$'', or ``zero temperature'', version of the game is that Alice starts at $v$, players take turns choosing the edges of a path, and the game is over as soon as one player chooses an edge to a vertex that has already been visited. Since the edge cover problem has feasible solutions regardless of the parity of the number of vertices, the zero temperature version already makes sense, but to simplify the analysis we can also introduce a finite $\theta$ version where a constraint may be violated (a player can quit the game) at cost $\theta/2$.

If the game is played on a tree (and for finite $\theta$ and large $n$ it will be, in view of the PWIT approximation), then the rules can be reformulated: The players take turns choosing the edges of a downward path in the tree, and a player has the right to quit at cost $\theta/2$ before their own move, or at zero cost after their move. The option of quitting at zero cost after their own move corresponds to deciding to cancel the opponent's next move. It makes no difference in theory if a player is required to make this decision before the opponent's move. 

The underlying philosophy when constructing the two-person game from the optimization problem is to think of an original game where Alice and Bob simply construct their solutions individually and then display them and compare the total lengths. The basic assumption is that Bob uses a particular edge $e$ which Alice does not use. Assuming optimal play, it clearly doesn't matter who reveals their solution first, or if the players take turns revealing parts of their solutions. Thus by applying principles of ``strategy stealing'', we successively transform the game into a series of other games which are equivalent under optimal play. Our goal, as designers of the game, is to decide the outcome of the game while having the players reveal as little as possible of their solutions.

Proceeding in analogy with our treatment of matching and TSP, we define yet another valuation concept on the PWIT by \begin{equation}\label{second} f(v) = \max(0, \min(\theta/2, l_i-f(v_i))).\end{equation} The crucial question of replica symmetry is whether for each $\theta$ there is (almost surely) only one valuation. If so, then $f(v)$ and $f(v_i)$ have the same distribution, and we can untangle the problem from an equation of ``self-consistency'' derived from \eqref{second}.  

\subsection{Establishing replica symmetry}
In the edge cover game, again only edges of length at most $\theta$ are relevant. We define $f_A$ and $f_B$ in the obvious way, and now $f_B(v)$ is always nonnegative. Therefore the relevant subset of the $l$-$f$-square is the parallelogram given by $0\leq f\leq \theta/2$ and $0\leq l-f\leq \theta/2$.

The crucial question is again the branching of near-optimal play given that, relative to $f_B$, Bob plays optimally and every decision by Alice is $\delta$-reasonable, where $\delta$ is a positive number that we can choose as a function of $d$ and $\theta$.

We want to carry out the strategy that we have used before, consisting in defining a certain quantity $H(k)$ that measures the branching of $R$ at level $k$. It turns out that we can use the same definition of $H(k)$  as for the matching problem, that is, according to \eqref{Hdef}, and that the calculations of Section~\ref{S:branching} then go through essentially unchanged. The edge cover game will terminate whenever it reaches a vertex of value either zero or $\theta/2$, but to establish that $f_A = f_B$ as in Section~\ref{S:branching} we actually only have to use the fact that it terminates at vertices of value $\theta/2$.

\subsection{The limit cost of the minimum edge cover}
In order to find the limit cost of the minimum edge cover, we let $\theta\to\infty$. The fact that we can interchange the limits $n\to\infty$ and $\theta\to\infty$ follows from the results of Section~\ref{S:perfectMatching}, and in fact the edge cover problem is considerably easier than matching in this respect since the analog of Poposition~\ref{P:completing} can be established using a simple greedy scheme for completing the diluted edge cover.

We can anticipate the final results of letting $\theta\to\infty$ by replacing \eqref{second} by \begin{equation} \label{first}  f(v) = \max(0, \min(\theta/2, l_i-f(v_i))),\end{equation} 
essentially assuming replica symmetry at zero temperature. Letting $F(x) = P(f(v)\geq x)$, we first notice that $F(x)=1$ when $x\leq 0$. For $x>0$, \eqref{first} translates into \begin{multline} F(x) = P(\text{no $l_i-f(v_i)<x$}) = P(\text{no $l_i$ with $f(v_i)>l_i-x$}) \\= \exp\left(-\int_0^\infty dl^{d-1}F(l-x)\dd l \right).\end{multline} 
Since $F(l-x)=1$ whenever $l\leq x$, it follows that \begin{multline} \label{coverIntegral} F(x) = \exp\left(-x^d-d\int_x^\infty l^{d-1}F(l-x)\dd l \right) \\= \exp\left(-x^d-d\int_0^\infty (t+x)^{d-1}F(t)\dd t \right).\end{multline}

\subsection{The case $d=1$}
When $d=1$, the integral equation \eqref{coverIntegral} takes the particularly simple form $$F(x) = e^{-A}\cdot e^{-x},$$ where $$A = \int_0^\infty F(t) \dd t$$ is independent of $x$. It follows that $A=e^{-A}$, and therefore $A=W(1)$ and $F(x) = W(1)\cdot e^{-x}$.  

In fact the zero temperature replica symmetry for $d=1$ can be verified directly from the dynamics of the integral operator that maps $F$ to the function $$x\mapsto \exp\left(-\int_0^\infty F(t)\dd t\right)\cdot e^{-x}.$$ This operator takes $A\cdot e^{-x}$ to $e^{-A}\cdot e^{-x}$, and the dynamics therefore only consists in iterating the function $A\mapsto e^{-A}$. It is easily verified that this converges to the fixed point $A=W(1)$ regardless of initial value of $A$.  
The limit cost is given by \begin{equation} \label{coverCost} \frac12 \int_0^\infty l\cdot P(f_1+f_2\geq l)\dd l,\end{equation} where $f_1$ and $f_2$ are independent and taken from the limit distribution. In terms of the function $F(x) = W(1)\cdot e^{-x}$, we have \begin{multline} P(f_1+f_2\geq l)\\ = (1-W(1))\cdot W(1)\cdot e^{-l} + \int_0^l W(1)\cdot e^{-x}\cdot W(1)\cdot e^{l-x}\dd x + W(1)\cdot e^{-l} \\= (1-W(1))\cdot W(1)\cdot e^{-l} + W(1)^2\cdot l \cdot e^{-l} + W(1)\cdot e^{-l} \\= 2W(1)\cdot e^{-l} - W(1)^2\cdot e^{-l} + W(1)^2\cdot le^{-l} .\end{multline}
It follows that \eqref{coverCost} is equal to $W(1) + W(1)^2/2\approx 0.72797$ in agreement with the result of \cite{HW09} for the bipartite graph.

\subsection{The case $d=2$}
Also for $d=2$ the problem allows a more or less explicit solution. For $d=2$, the integral operator is given by $$\exp\left(-x^2-2\int_0^\infty (t+x)F(t)\dd t\right) = \exp\left(-x^2-2Ax -2B\right),$$ where $$A = \int_0^\infty F(t)\dd t$$ and $$B= \int_0^\infty tF(t)\dd t.$$
 
This operator acts on the 2-dimensional space of $A$ and $B$ by $$\begin{pmatrix}A\\B\end{pmatrix} \mapsto \begin{pmatrix} -\frac12\sqrt{\pi}\cdot e^{-2B+A^2}(\mathrm{erf}(A)-1)\\ \frac12e^{-2B}+ \frac12A\sqrt{\pi}e^{-2B+A^2}(\mathrm{erf}(A)-1)\end{pmatrix}$$
At the fixed point we must have $$B = \frac12e^{-2B} - A^2,$$ which makes it easy to numerically find $A\approx 0.41079$ and $B\approx 0.18005$. We get $$P(f_1+f_2\geq l) = (1-F(0))F(l)+\int_0^l(2x+2A)F(x)F(l-x)\dd x,$$ and the limit cost is $$\int_0^\infty l^2\cdot P(f_1+f_2\geq l)\dd l \approx 0.55872.$$
 
\section{Concluding remarks} \label{S:variations}
There are several slightly different random models and optimization problems that yield to similar analysis. If the underlying graph is the complete bipartite graph $K_{n,n}$ instead of the complete graph $K_n$, then a slight change is needed in the estimates in Section~\ref{S:PWITlike}, but otherwise the only differences are trivial scaling factors. For finite $\theta$ it is possible to take the generalization further to various forms of near-regular graphs, but then it is not clear to what extent the $\theta\to\infty$ limit corresponds to a perfect matching or tour.

If the edge-lengths are independent and taken from some distribution satisfying $P(l<r)\sim r^d$ for small $r$, then one can show using fairly standard arguments that our main theorems about convergence in probability of the average length of the edges in the solution still hold, although the expected total length of the solution may not even exist.

It is also possible to generalize our results to other optimization problems, and this seems to lead to a number of open-ended questions. Clearly the results presented in this paper generalize to requiring each vertex to have a fixed degree other than 1 or 2, but there is also a large family of other problems, including minimum edge cover, where the $d=1$ case can be analyzed with the methods of \cite{W09}. It would be interesting to see to what extent one can determine their asymptotics for $d\neq 1$. 

There are several questions that we have left unanswered. Apart from the obvious question whether our results hold also for $0<d<1$, we have already mentioned the question whether the operators $V_\infty$ and $W_\infty$ have unique fixed points that are the limits of the fixed points of $V_\theta$ and $W_\theta$.

Another issue that we have not discussed is the Asymptotic Essential Uniqueness (AEU) of the optimum solutions. The AEU property \cite{A01} is essential for questions about the distribution of edge-lengths and the nearest neighbor statistics for perfect matching and the TSP. The results of \cite{PR01} can probably be verified in the $\theta\to\infty$ limit with our methods, but concluding that they are valid for perfect matching and the TSP requires interchanging the order of the limits $n\to\infty$ and $\theta\to\infty$, which would be justified if the AEU property was established.
  
We also have not discussed any algorithmic aspects. It has been recognized for a long time in the physics community that for several optimization problems, replica symmetry suggests that distributed iterative schemes like Belief Propagation and Simulated Annealing are likely to be successful. We have already mentioned the paper \cite{SS09} which analyzes a Belief Propagation algorithm for minimum matching, and higher capacity problems are studied in \cite{BBCZ08}. A natural suggestion in view of our results and in analogy with simulated annealing is to solve the diluted matching problem using Belief Propagation and successively increase the value of $\theta$.

\end{document}